\documentclass[12pt]{amsart}

\usepackage[english]{babel}

\usepackage{fullpage}

\usepackage{amsfonts,color,newclude}
\usepackage{amsthm}
\usepackage{amssymb,transparent}
\usepackage{amsmath}
\usepackage{amscd}
\usepackage{thmtools}
\usepackage{graphicx}
\usepackage{tikz}
\usepackage{comment}

\theoremstyle{definition}
\newtheorem{theorem}{Theorem}
\newtheorem{defn}[theorem]{Definition}
\newtheorem{cor}[theorem]{Corollary}
\newtheorem{question}[theorem]{Question}

\newtheorem{prop}[theorem]{Proposition}

\newcommand{\VH}{\mathcal{VH}}

\title{Cleanliness versus Specialness}
\author{Kasia Jankiewicz}
\address{Department of Mathematics\\
			University of California\\
			Santa Cruz, USA}
	\email{kasia@ucsc.edu}

\begin{document}
\maketitle

\begin{abstract}
We show that the fundamental group of a geometrically clean graph of finite rank free groups does not need to be virtually compact special, answering a question of Wise. This implies that the class of the virtually $\VH$-clean graphs of finite rank free groups is a proper subclass of the class of virtually geometrically clean graphs of finite rank free groups.
\end{abstract}

\section{Introduction and background}
We assume the reader's familiarity with the notions of graphs of spaces and graphs of groups, and refer to \cite{ScottWall} for the background.

Let $\Gamma$ be a directed graph, possibly with loops and multiple edges. Let $X(\Gamma)$ be a \emph{graph of graphs} (i.e.\ a graph of spaces whose vertex and edge spaces are graphs), with vertex graph $X_v$ for $v\in V(\Gamma)$, edge graphs $X_e$ for $e\in E(\Gamma)$, and edge maps $f^{\iota}_e: X_e\to X_{\iota(e)}$ and $f^\tau_e: X_e \to X_{\tau(e)}$, where $\iota(e),\tau(e)$ denote the initial and terminal vertex of the edge $e$.

Let $G(\Gamma)$ be the associated graph of groups, i.e.\ the vertex groups are $G_v = \pi_1(X_v)$ for $v\in V(\Gamma)$, the edge groups are $G_e = \pi_1(X_e)$ for $e\in E(\Gamma)$, and the edge maps are $\phi^{\iota}_e = (f^{\iota}_e)_*: G_e\to G_{\iota(e)}$ and $\phi^{\tau}_e = (f^\tau_e)_*: G_e \to G_{\tau(e)}$. 
If all $X_v, X_e$ are finite graphs, then all $G_v, G_e$ are finite rank free groups. In such a case, we say that $G(\Gamma)$ is a \emph{graph of finite rank free groups}.

Recall that a map $f:X\to Y$ of graphs is \emph{combinatorial} if it maps vertices to vertices, and each open edge is mapped homeomorphically onto an open edge.
\begin{defn}
   A graph of graphs $X(\Gamma)$ is 
   \begin{itemize}
   \item \emph{$\VH$-clean} if each map $f^{\iota}_e$ and $f^{\tau}_e$ is a combinatorial embedding,
   \item\emph{geometrically clean} if each map $f^{\iota}_e$ and $f^{\tau}_e$ is an embedding,
   \item\emph{algebraically clean} if each homomorphism $\phi^{\iota}_e$ and $\phi^{\tau}_e$ is an injection onto a free factor.
   \end{itemize}
   Similarly, we say that a graph of free groups $G(\Gamma)$ is clean in one of the senses above, if it is associated with a graph of graphs with the same property. As usual, we say $G(\Gamma)$ is \emph{virtually} clean in one of the senses above, if its fundamental group has a finite index subgroup that splits as a graph of graphs with that property.
\end{defn}
Clearly, every $\VH$-clean graph of groups is geometrically clean, and every geometrically clean graph of groups is algebraically clean. One can easily construct examples of graphs of groups which are geometrically clean but not $\VH$-clean, or algebraically clean but not geometrically clean. However, a more interesting question arises when we consider all possible splittings of a given group as a graph of finite rank free groups, and allow to pass to finite index subgroups.

\begin{question}\label{question}
Are the classes of groups that virtually splits as
\begin{itemize}
\item $\VH$-clean graphs of finite rank free groups,
\item geometrically clean graphs of finite rank free groups,
\item algebraically clean graphs of finite rank free groups
\end{itemize}
distinct?
\end{question}

The total space of a $\VH$-clean graph of graphs can be realized asa clean \emph{$\VH$ square complexes}, as first introduced by Bridson-Wise \cite{BridsonWise} ($\VH$ stands for \emph{vertical} and \emph{horizontal}). Wise further studied and utilized the above notions of the cleanliness in \cite{WiseFigure8, WisePolygons}. In particular, in \cite{WisePolygons} he gave the definitions of algebraically clean (referred to as ``clean'') and geometrically clean graphs of groups, and showed that the fundamental groups of algebraically clean graphs of finite rank free groups are residually finite.

Next, Haglund-Wise showed that $\VH$-cleanliness implies virtual specialness.
\begin{theorem}[{\cite[Thm 5.7]{HaglundWiseSpecial}}]
    Every group which splits as a $\VH$-clean graph of finite rank free groups is virtually compact special.
\end{theorem}

More recently, the author and Schreve showed that the fundamental groups of  algebraically clean graphs of finite rank free groups are also cohomologically good, and for each prime $p$ virtually residually $p$-finite \cite{JankiewiczSchreveProfiniteAlgebraicallyClean}. The author has also showed that many Artin groups virtually split as algebraically clean graphs of finite rank free groups \cite{JankiewiczArtinRf, JankiewiczArtinSplittings}, however many of them are not virtually compact special \cite{HJP, HaettelArtin}. Thus the class of groups which virtually split as $\VH$-clean graphs of finite rank free groups is properly contained in the class of groups which virtually split as algebraically clean graphs of finite rank free groups, which partially answer Question~\ref{question}.

Wise asked\footnote{Private communication.} whether every geometrically clean graph of finite rank free groups is virtually compact special. We answer this question negatively. 

\begin{theorem}\label{thm: existence}
    There exists a geometrically clean graph of finite graphs whose fundamental group is not virtually compact special.
\end{theorem}

An example of a group satisfying Theorem~\ref{thm: existence} is the group $G$ which is the fundamental group of the graph of graphs $G(\Theta)$ constructed in Section~\ref{sec: the graph of groups} and illustrated in Figure~\ref{fig: graph of groups}. Algebraically, this group is a free product of two copies of the free group $F_4$ of rank four, amalgamated three times, each time along a copy of $F_3$. 
\begin{figure}
    \centering
    \includegraphics{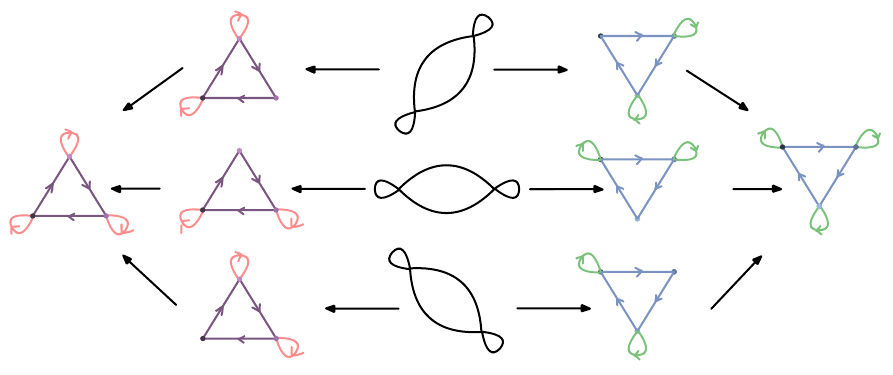}
    \caption{A geometrically clean graph of graphs $X(\Theta)$, whose fundamental group $G(\Theta)$ is not virtually compact special.}
    \label{fig: graph of groups}
\end{figure}
In Section~\ref{sec: the graph of groups}, we also show that this group embeds as a finite index subgroup in a certain Artin group, see Proposition~\ref{prop: finite index subgroup}. This embedding is similar to the proof that many Artin groups are virtually algebraically clean given in \cite{JankiewiczArtinRf, JankiewiczArtinSplittings}. However, there we never explicitly construct algebraically clean graphs of finite rank free groups, only argue that such covers must exist. For the groups considered in \cite{JankiewiczArtinRf, JankiewiczArtinSplittings}, the algebraically clean graphs fail to be geometrically clean.

As a corollary we obtain a further partial answer to Question~\ref{question}.
\begin{cor}\label{cor: properly contained}
    The class of groups which virtually split as $\VH$-clean graphs of finite rank free groups is properly contained in the class of groups which virtually split as geometrically clean graphs of finite rank free groups.
\end{cor}

The following remains open. 
\begin{question}
    Does there exist a group that splits as an algebraically clean graph of finite rank free groups, and that does not virtually split as a geometrically clean graph of finite rank free groups?
\end{question}

\section{The construction of the graph of groups $G(\Theta)$}\label{sec: the graph of groups}

Let $\Theta$ be the theta-graph, i.e.\ a graph with two vertices $v, w$ and three edges $e_1, e_2, e_3$, each with initial vertex $v$ and the terminal vertex $w$.
Let $X(\Theta)$ be a graph of graphs illustrated in Figure~\ref{fig: graph of groups}. I.e.\ each vertex graph $X_v, X_w$ is a triangle (length $3$ cycle) with a loop attached to each vertex, and each edge group $X_{e_i}$ is a bigon (length $2$ cycle) with a loop attached to each vertex. The edge maps $f^{\iota}_{e_i}$ and $f^{\tau}_{e_i}$ are (non-combinatorial) embeddings, as illustrated in Figure~\ref{fig: graph of groups}. Let $G(\Theta)$ be the associated graph of groups, i.e. $G(\Theta)$ is the following graph of groups 
\begin{center}
    \begin{tikzpicture}
    \node at (-1.2,0) {$G(\Theta)=$};
    \node at (0,0) (v) {$F_4$};
    \node at (2,0) (w) {$F_4$};
    \draw (v) to[out = 90, in = 90] (w);
    \draw (v) to (w);
    \draw (v) to[out = -90, in = -90] (w);
    \node at (1,0.25) {$F_3$};
    \node at (1,1.15) {$F_3$};
    \node at (1,-0.65) {$F_3$};
    \end{tikzpicture}
\end{center}
where the inclusions of the copies of $F_3$ in $F_4$ are given by the inclusions of graphs in Figure~\ref{fig: graph of groups}. Let $G$ be the fundamental group of $G(\Theta)$.

The following is immediate.
\begin{prop}
        The graph of graphs $X(\Theta)$ (and consequently also the graph of groups $G(\Theta)$) is geometrically clean.
\end{prop}

Let $A_{2,3,\infty}$ be the Artin group with defining graph a path of length two with labels $2$ and $3$, i.e.
\[
A_{2,3,\infty} = \langle a,b,c \mid ab=ba, bcb=cbc\rangle.
\]

\begin{prop}\label{prop: finite index subgroup}
    The group $G$ embeds as an index $6$ subgroup in the Artin group $A_{2,3,\infty}$.
\end{prop}

\begin{proof}
First we rewrite the standard presentation of $A_{2,3,\infty}$.
\[
A_{2,3,\infty} = \langle a,b,x \mid ab = ba, bxb = x^2\rangle.
\]
Indeed, by setting $x = cb$, the relation $bxb = x^2$ is equivalent to $bcb = cbc$. The presentation complex $Y$ of the new presentation of $A_{2,3,\infty}$ is on the left side of Figure~\ref{fig: double cover}.
\begin{figure}
    \centering
    \includegraphics{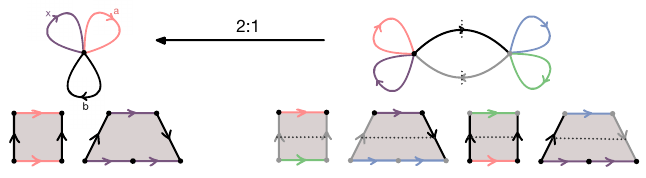}
    \caption{On the left: a presentation complex of $A_{2,3,\infty}$. On the right: its double cover whose fundamental group is the kernel of a homomorphism to $\mathbb Z/2$.}
    \label{fig: double cover}
\end{figure}
Consider the homomorphism $h: A_{2,3,\infty}\to \mathbb Z/2$ sending $a,x$ to $0$, and $b$ to $1$. The kernel $\ker h$ is an index $2$ subgroup of $A_{2,3,\infty}$ whose presentation complex $\hat Y$ is the double cover of $Y$ illustrated on the right side of Figure~\ref{fig: double cover}.

The complex $\hat Y$ can be realized as the total space of the graph of spaces as follows. 
The underlying graph is a single edge. The vertex spaces are (1) the wedge of the pink and purple loops, denoted by $Y_{v_1}$, and (2) the wedge of the green and blue loops, denoted by $Y_{v_2}$. The edge space $Y_e$ is the dotted horizontal graph on the right side of Figure~\ref{fig: double cover}, which is homeomorphic to a bigon with a loop at each vertex.
Indeed, the complex $\hat Y$ is obtained as  $(Y_{v_1} \cup Y_e\times [0,1]\cup Y_{v_2})/\sim$ where $(y,0)\sim f_e^{v_1}(y)$ and $(y,1)\sim f_e^{v_2}(y)$ for $y\in Y_e$, with the gluing maps $f_e^{v_1}, f_e^{v_2}$ illustrated in Figure~\ref{fig: amalgamated product}.
\begin{figure}
    \centering
    \includegraphics{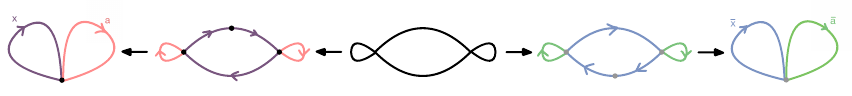}
    \caption{The group $\pi_1 \hat Y$ splits as an amalgamated product $F_2*_{F_3}F_2$ with the two maps $F_3\to F_2$ induced by the graph maps above.}
    \label{fig: amalgamated product}
\end{figure}
The black and gray edges of $\hat Y$ (which each map in $Y$ to the loop labelled $b$) are contained in $Y_e\times [0,1]$ as the product of one of the vertices of $Y_e$ with the interval.

Algebraically, this yields a splitting of $\pi_1 \hat Y$ as an amalgamated product $F_2*_{F_3}F_2$.
The amalgamating subgroup $F_3$ maps to the first copy of $F_2 = \langle a,x \rangle$ as $\langle x^3, a, x^{-1}ax\rangle$, and to the second copy of $F_2 = \langle \bar a,\bar x \rangle$ as $\langle \bar x^3, \bar a, \bar x \bar a\bar x^{-1}\rangle$, and the identification of these two $F_3$ subgroups is via an isomorphism sending $x^3 \mapsto \bar x^3, a\mapsto \bar a, x^{-1}ax \mapsto \bar x^{-2} \bar a\bar x^{2}$.

In particular, a homomorphism $k:\pi_1 \hat Y\to \mathbb Z/3$ sending $x,\bar x\to 1$ and $a,\bar a\to 0$ is well-defined, and constant while restricted to the amalgamating subgroup.
The kernel $\ker k$ is an index $3$ subgroup of $\pi_1 \hat Y$ which splits as a graph of groups $G(\Theta)$, and the corresponding triple cover of $\hat Y$ splits as a graph of graphs $X(\Theta)$, illustrated in Figure~\ref{fig: graph of groups}. This proves that the fundamental group $G$ of $G(\Theta)$ embeds as an index $6$ ($=2\cdot 3$) subgroup of $A_{2,3,\infty}$.

\end{proof}

By \cite{HJP} or \cite{HaettelArtin} $A_{2,3,\infty}$ is not virtually cocompactly cubulated, so in particular not virtually cocompact special. The group $G$, as a finite index subgroup of $A_{2,3,\infty}$ by Proposition~\ref{prop: finite index subgroup}, has the same properties.
\begin{cor}
    The group $G$ is not virtually cocompactly cubulated, so in particular not virtually compact special.
\end{cor}

\section{An infinite family of 
examples}
Any finite index subgroup of the group $G$ constructed in the previous section splits as a geometrically clean graph of finite rank free groups, and is not virtually compact special, hence not virtually $\VH$-clean.

More examples arise from Artin groups $A_{2,n,\infty}$ for odd $n\geq 3$. These
are also not virtually compact special \cite{HJP}. 
The construction presented in the previous section generalizes to all Artin group $A_{2,n,\infty}$ for odd $n\geq 3$, and yields an index $2n$ subgroup of $A_{2,n,\infty}$ which is a geometrically clean graph of free groups. The underlying graphs is  a $\Theta_n$-graph (i.e.\ a graph on two vertices with $n$ edges joining the two vertices) of groups, with vertex groups isomorphic to $F_{n+1}$, and all $n$ edge groups isomorphic to $F_3$. See Figure~\ref{fig: n=7} for the resulting graph of groups where $n=7$.
\begin{figure}
    \centering
    \includegraphics{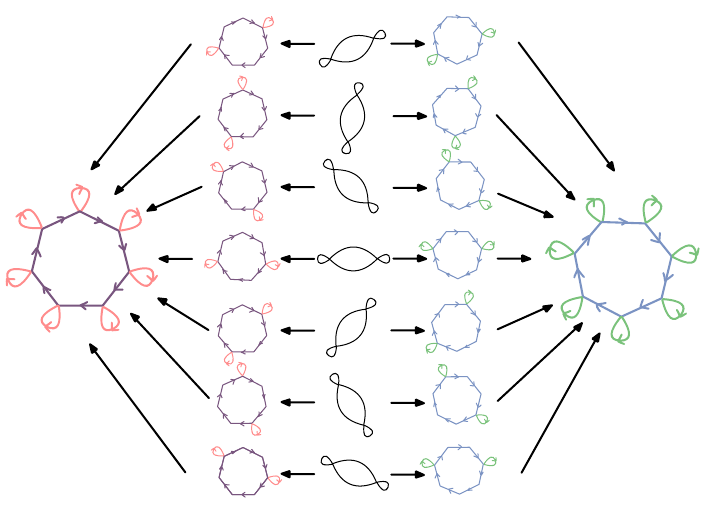}
    \caption{Geometrically clean graph of free groups, whose vertex group are copies of $F_8$, and edge groups are copies of $F_3$, whose fundamental group embeds as an index $14$ subgroup of the Artin group $A_{2,7,\infty}$.}
    \label{fig: n=7}
\end{figure}

\section*{Acknowledgements} The author thanks Sam Fisher for asking her whether geometrically clean graphs of free groups are special, and Zach Munro for a correction in the earlier version. This material is based upon work supported by the National Science Foundation
under Grants No. DMS-1926686 and DMS-2238198.

\bibliographystyle{alpha}
\bibliography{sample}

\end{document}